\documentclass[a4paper,12pt,reqno]{amsart}

\usepackage[english]{babel}
\usepackage[T1]{fontenc}
\usepackage[left=1.1in,right=1.1in,top=1.2in,bottom=1.2in]{geometry}
\usepackage{times}
\usepackage{microtype}

\usepackage{commath,amssymb,amscd,mathrsfs,mathtools,dsfont,bbm,multirow,array,caption}
\usepackage[all]{xy}
\usepackage{enumerate}
\usepackage{longtable}
\usepackage{comment}

\usepackage[usenames,dvipsnames,svgnames,table]{xcolor}
\definecolor{red}{RGB}{255,25,25}
\definecolor{blue}{RGB}{25,50,200}
\usepackage{tikz-cd}
\usetikzlibrary{decorations.pathmorphing}

\usepackage[pagebackref,linktocpage,dvipdfmx]{hyperref}

\hypersetup{
pdftitle={\@title},			
pdfauthor={\authors},		
colorlinks=true,				
linkcolor=red,				
citecolor=MidnightBlue,		
filecolor=magenta,			
urlcolor=MidnightBlue			
}

\usepackage{cleveref} 

\newtheorem{theorem}{Theorem}[section]
\crefname{theorem}{Theorem}{Theorems}
\newtheorem{lemma}[theorem]{Lemma}
\crefname{lemma}{Lemma}{Lemmas}

\crefname{proposition}{Proposition}{Propositions}
\newtheorem{prop}[theorem]{Proposition}
\crefname{prop}{Proposition}{Propositions}

\crefname{corollary}{Corollary}{Corollaries}
\newtheorem{cor}[theorem]{Corollary}
\crefname{cor}{Corollary}{Corollaries}

\crefname{conjecture}{Conjecture}{Conjectures}
\newtheorem{conj}[theorem]{Conjecture}
\crefname{conj}{Conjecture}{Conjectures}
\newtheorem*{conj*}{Conjecture}
\crefname{conj}{Conjecture}{Conjectures}

\theoremstyle{definition}

\crefname{definition}{Definition}{Definitions}
\newtheorem{defn}[theorem]{Definition}
\crefname{defn}{Definition}{Definitions}

\crefname{example}{Example}{Examples}

\crefname{notation}{Notation}{Notation}
\newtheorem*{notation*}{Notation}
\crefname{notation}{Notation}{Notation}

\crefname{problem}{Problem}{Problems}

\crefname{question}{Question}{Questions}

\crefname{condition}{Condition}{Conditions}

\crefname{assumption}{Assumption}{Assumptions}

\newtheorem{hyp}{}

\theoremstyle{remark}
\newtheorem{rmk}[theorem]{Remark}
\crefname{rmk}{Remark}{Remarks}
\newtheorem*{rmk*}{Remark}
\crefname{rmk}{Remark}{Remarks}

\crefname{remark}{Remark}{Remarks}

\crefname{fact}{Fact}{Facts}

\crefname{claim}{Claim}{Claims}
\newtheorem*{claim*}{Claim}
\crefname{claim}{Claim}{Claims}

\crefname{step}{Step}{Steps}

\crefname{case}{Case}{Cases}

\numberwithin{equation}{section}
\renewcommand{\emptyset}{\varnothing}

\newcommand{\lra}{\longrightarrow}

\newcommand{\arxiv}[1]{\href{https://arxiv.org/abs/#1}{{\tt arXiv:#1}}}

\newcommand{\bbP}{\mathbb{P}}
\newcommand{\bbQ}{\mathbb{Q}}
\newcommand{\bbR}{\mathbb{R}}
\newcommand{\bbZ}{\mathbb{Z}}

\newcommand{\bQ}{\mathbb{Q}}
\newcommand{\bR}{\mathbb{R}}
\newcommand{\bZ}{\mathbb{Z}}

\newcommand{\Aut}{\operatorname{Aut}}

\newcommand{\Eff}{\overline{\operatorname{Eff}}}

\newcommand{\id}{\operatorname{id}}

\newcommand{\isom}{\simeq}

\newcommand{\Nef}{\operatorname{Nef}}

\newcommand{\Pic}{\operatorname{Pic}}

\newcommand{\Supp}{\operatorname{Supp}}

\begin{document}

\title[Kawaguchi-Silverman conjecture]{A note on  Kawaguchi-Silverman conjecture}
\author{Sichen Li}
\address{
School of Mathematical Sciences, Fudan University, Shanghai 200433, People's Republic of China}
\email{\href{mailto:lisichen123@foxmail.com}{lisichen123@foxmail.com}}
\urladdr{\url{https://www.researchgate.net/profile/Sichen_Li4}}
\author{Yohsuke Matsuzawa}
\address{Department of Mathematics, Box 1917, Brown University, Providence, Rhode Island 02912, USA}
\email{\href{mailto:matuzawa@math.brown.edu}{matsuzawa@math.brown.edu}}
\begin{abstract}
We collect some results on endomorphisms on projective varieties related with the Kawaguchi-Silverman conjecture.
We discuss certain condition on automorphism groups of projective varieties and
positivity conditions on leading real eigendivisors of self-morphisms.
We prove Kawaguchi-Silverman conjecture for endomorphisms on projective bundles on
a smooth Fano variety of Picard number one.
In the last section, we discuss endomorphisms and augmented base loci of their eigendivisors.
\end{abstract}





\maketitle
\section{Introduction}

Let $X$ be a smooth projective variety of dimension $n\ge1$ defined over $ \overline{\mathbb Q}$.
Let $f: X\dashrightarrow X$ be a dominant rational map, 
and $f^*: N^{1}(X)_{\bR} \to N^{1}(X)_{\bR}$ be the induced map, where
$N^{1}(X)_{\bR}:=N^{1}(X) {\otimes}_{\bZ}\bR$ and $N^{1}(X)$ is the group of Cartier divisors modulo numerical equivalence.
Let $\rho(T,V)$ denote the spectral radius of a linear transformation $T: V\to V$ of a real or complex vector space.
Then the {\it first dynamical degree of  $f$} is the quantity
\begin{equation*}
                                   \delta_f:=\lim_{m\to\infty} \rho((f^m)^*,N^{1}(X)_{\bR})^{1/m}.
\end{equation*}
Alternatively, if we let $H$ be any ample divisor on $X$, then $\delta_f$ is also given by the formula
\begin{equation*}
                                  \delta_f=\lim_{n\to \infty}((f^m)^{*}H\cdot H^{n-1})^{1/m}.
\end{equation*}
By definition, it is easy to see that $\delta_{f^r}=(\delta_f)^r$.
Dynamical degree is actually a birational invariant and therefore we can define it for dominant rational maps on possibly singular projective varieties
by taking resolution of singularities.
For basic properties of dynamical degree, see  \cite{dang, DF, tru0}.

In \cite{KS16-tams,KS16}, Kawaguchi and Silverman studied an analogous arithmetic degree, which we now describe.
Assume that $X$ and $f$ are defined over $\overline{\mathbb Q}$, and write $X(\overline{\mathbb Q})_f$ for the set of points $x$ whose forward $f$-orbit
\begin{equation*}
                                    \mathcal O_f(x)=\big\{ x, f(x), f^2(x),\cdots\big\}
\end{equation*}
is well defined (i.e. $f^{m}(x)$ is not contained in the indeterminacy locus of $f$ for all $m\geq0$).
Further, let
\begin{equation*}
                        h_X: X(\overline{\mathbb Q})\to [0,\infty)
\end{equation*}
be a logarithmic Weil height function on $X$ associated with an ample divisor, and let $h_X^+=\max\big\{1, h_X\big\}$.
The {\it arithmetic degree of f at $x\in X(\overline{\mathbb Q})_f$} is the quantity
\begin{align*}
  \alpha_f(x)=\lim_{n\to\infty} h_X^+(f^n(x))^{1/n},
\end{align*}
if the limit exists.

The following Kawaguchi-Silverman conjecture (KSC for short) asserts that for a dominant rational self-map $f: X\dashrightarrow X$ of a projective variety $X$ over $\overline{\mathbb Q}$, the arithmetic degree $\alpha_f(x)$ of any point $x$ with Zariski dense $f$-orbit is equal to the first dynamical degree $\delta_f$ of $f$.
\begin{conj}[KSC]
(cf. \cite[Conjecture 6]{KS16}\label{KSC})
Let $f: X\to X$ be a dominant rational self-map of a projective variety $X$ over $\overline{\bQ}$, and let $x\in X(\overline{\mathbb Q})_f$.
Then the following hold.
\begin{enumerate}
\item[(1)] The limit  defining $\alpha_f(x)$ exists.
\item[(2)] If $\mathcal O_f(x)$ is Zariski dense in $X$, then $\alpha_f(x)=\delta_f$.
\end{enumerate}
\end{conj}

\begin{rmk}
There are many special cases that this conjecture is proved.
For recent results, see for example \cite{LS18-IMRN,Matsuzawa19,MY19,MZ19}.
\end{rmk}

\begin{rmk}\label{NZ09}
If $X$ has positive Kodaira dimension, then any dominant rational map does not have any Zariski dense orbits.
This follows from ''finiteness of pluricanonical representation''.
See \cite[Remark 1.2]{MSS18}, \cite[Theorem 14.10]{Ueno75} or \cite[Theorem A]{NZ09}
Therefore, \cref{KSC} (2) is meaningful only for projective varieties of nonpositive Kodaira dimension.
\end{rmk}

\begin{rmk}
If $f$ is a morphism, then the existence of arithmetic degree is proved, i.e.  \cref{KSC}(1) is true \cite{KS16-tams}.
\end{rmk}

In this paper, we focus on endomorphisms on projective varieties and collect some results related to this conjecture.
In \cref{Preliminaries}, we list some basic facts on KSC.
In \cref{reduction}, we show some reduction results on KSC.
In \cref{projbdl}, we prove KSC for endomorphisms on projective bundles over a smooth Fano variety of Picard number one.
In \cref{canht}, we give positivity conditions on leading eigendivisors that is enough to prove KSC.
In \cref{augb}, we discuss the augmented base loci of leading eigendivisors. This section is less related with KSC.

Throughout this paper, the ground field is $ \overline{\mathbb Q}$ unless otherwise stated.

\section{Preliminaries}\label{Preliminaries}

We gather some facts on height functions.
See \cite{bg, hs, Lan} for the definition and basic properties of height functions.
Here, we simply list some fundamental facts that will be used in this paper.
\begin{itemize}
\item $h_E$ is bounded below outside $\mathrm{Supp}~E$ for any effective Cartier divisor $E$.
\item $h_{\sum a_iD_i}=\sum a_ih_{D_i}+O(1)$ where $O(1)$ is a bounded function.
\item Let $\pi: X\to Y$ be a surjective morphism of normal projective varieties and $B$ an $\bR$-Cartier divisor on $Y$.
Then $h_B(\pi(x))=h_{\pi^*B}(x)+O(1)$ for any $x\in X(\overline{\mathbb Q})$.
\end{itemize}

We list several basic facts on KSC which will be used in the rest of the paper.

\begin{lemma}\label{generically finite}
\text{(cf. \cite[Lemma 2.5]{MZ19})}
Let $\pi: X\dashrightarrow Y$ be a dominant rational map of projective varieties.
Let $f:X\to X$ and $g:Y\to Y$ be surjective endomorphisms such that $g\circ\pi=\pi\circ f$.
Then the following hold.
\begin{enumerate}
\item[(1)] Suppose $\pi$ is generically finite.
Then KSC holds for $f$ if and only if KSC holds for $g$.
\item[(2)] Suppose $\delta_f=\delta_g$ and KSC holds for $g$.
Then KSC holds for $f$.
\end{enumerate}
\end{lemma}

\begin{lemma}\label{Sano}
\text{(cf. \cite[lemma 3.2]{Sano16}, \cite[Lemma 3.3]{Silverman17})}
Let $f: X\to X$ and $g: Y\to Y$ be two surjective morphisms of  projective varieties.
Suppose KSC holds for both $f$ and $g$.
Then KSC holds for $f\times g$.
\end{lemma}

\begin{lemma}\text{(cf, \cite[Remark 2.9]{Matsuzawa19})}
Let $f \colon X \longrightarrow X$ be a surjective endomorphism on a projective variety $X$.
Let $n$ be a positive integer.
Then KSC for $f$ is equivalent to KSC for $f^{n}$.
\end{lemma}

A normal projective variety $X$ is called {\it Q-abelian} if there is an abelian variety $A$ and a finite surjective morphism $A\to X$ which is quasi-\'etale,
i.e. \'etale in codimension one.
\begin{theorem}\label{Q-abelian}
\text{(cf. \cite[Theorem 2.8]{MZ19})}
Let $X$ be a $Q$-abelian variety.
Then KSC holds for any surjective endomorphism of $X$.
\end{theorem}

\section{Reductions of KSC}\label{reduction}

\subsection{Automorphism groups}
We refer to Koll\'ar-Mori \cite{KM98} for standard notions and terminologies in birational geometry.
A normal projective variety $X$ is called weak Calabi-Yau variety if 
$X$ has at most canonical singularities, $K_{X}\sim 0$, and the augmented irregularity 
\[
 \widetilde{q}(X):=\sup \{h^{1}(Y, \mathcal{O}_{Y}) \mid \text{$Y \to X$ finite surjective quasi-\'etale}\}
\]
is zero.

\begin{prop}\label{wcy}
Let $X$ be a normal projective variety with at most klt singularities of dimension $n\ge1$.
Then \cref{KSC} is true for all automorphisms of normal projective varieties with dimension at most $n$ and $K_X$ is numerically trivial if and only if \cref{KSC} is true for all automorphisms of weak Calabi-Yau varieties with dimension at most $n$.
\end{prop}

\begin{rmk}
This is a generalization of \cite[Corollary 1.4]{LS18-IMRN} to singular case.
\end{rmk}

\begin{rmk}
Let $f$ be an automorphism of a normal projective variety $X$.
There is a $G$-equivariant resolution $\pi: X'\to X$ and an automorphism $f'$ of $X'$ such that $\pi\circ f'=f\circ\pi$ (cf. \cite[Theorem 13.2]{BM97}).
Therefore, KSC for $f$ reduces to KSC for $f'$, which is an automorphism on a smooth projective variety.
This makes problem easier sometime, but it is sometime better to work on singular variety because it might have better birational geometric properties.
\end{rmk}

We use the following Kawamata-Nakayama-Zhang's weak decomposition theorem (cf. \cite{Kawamata85,NZ10},\cite[Lemma 2.7]{HL19})).

\begin{lemma}
\label{lemma:lifting}
Let $X$ be a normal projective variety with at most klt singularities such that $K_X \sim_\bQ 0$, and $f$ an automorphism of $X$.
Then there exist a morphism $\pi \colon \widetilde{X} \lra X$ from a normal projective variety $\widetilde{X}$,
an automorphism $\widetilde{f}$ of $\widetilde{X}$ such that the following conditions hold.
\begin{itemize}
\item[(1)] $\pi$ is finite surjective and \'etale in codimension one.
\item[(2)] $\widetilde{X}$ is isomorphic to the product variety $Z \times A$  for a weak Calabi--Yau variety $Z$ and an abelian variety $A$.
\item[(3)] The dimension of $A$ equals the augmented irregularity $\widetilde{q}(X)$ of $X$.
\item[(4)] There are automorphisms $\widetilde{f}_Z$ and $\widetilde{f}_A$ of $Z$ and $A$, respectively, such that the following diagram commutes:
\[
\xymatrix{
X \ar[d]_{f} & & \widetilde{X} \ar[ll]_{\pi} \ar[d]_{\widetilde{f}} \ar[rr]^{\isom} & & Z \times A \ar[d]^{\widetilde{f}_Z \times \widetilde{f}_A} \\
X & & \widetilde{X} \ar[ll]_{\pi} \ar[rr]^{\isom} & & Z \times A .
}
\]
\end{itemize}
\end{lemma}

\begin{proof}[Proof of \cref{wcy}]
This follows from abundance for numerically trivial canonical divisor, \cref{lemma:lifting},
\cref{generically finite}(1), \cref{Sano}, and  \cref{Q-abelian}.
\end{proof}

Certain conditions on automorphism groups of a variety make us possible to run MMP equivariantly and 
reduce KSC to that of special varieties.
For a normal projective variety $X$ of dimension $n$ and a subgroup $G$ of the automorphism group $\Aut(X)$ of $X$,
consider the following condition:

\begin{hyp}[$n,r$] \label{hyp}
$G \isom \bZ^{r}$ with $1 \le r \le n-1$ and  $G$ is of positive entropy, i.e., $\delta_{g}>1$ for all $g\in G \setminus \{\id\}$.
\end{hyp}

\begin{rmk}
In \cite{DS04}, they prove that every commutative subgroup $G$ of $\Aut(X)$ has rank at most $\dim X-1$.
\end{rmk}

\begin{theorem}\label{MainProp}
Let $X$ be a normal projective variety of dimension $n$ with at most klt singularities and $G$ be a subgroup of $\Aut(X)$.
Then the following statements hold.
\begin{enumerate}
\item Suppose  $X$ and $G$ satisfies {\em \cref{hyp}}$(n, n-1)$ and $X$ is not rationally connected.
Then \cref{KSC} is true for all automorphisms $g\in G$.

\item Suppose $K_X\equiv 0$ and $(X, G)$ satisfies {\em \cref{hyp}}$(n, n-2)$.
Then \cref{KSC} is true for all automorphisms $g\in G$ if  \cref{KSC} is true for all automorphisms of weak Calabi-Yau varieties of dimension $n$.
\end{enumerate}
\end{theorem}

\begin{proof}[Proof of \cref{MainProp}]
(1) Since $X$ and $G$ satisfies {\em \cref{hyp}}$(n, n-1)$ and $X$ is not rationally connected, then by \cite[Theorems 1.1 and 2.4]{Zhang16} or \cite[Theorem 1.1]{HL19}, after replacing $G$ by a finite-index subgroup, $X$ is  $G$-equivariantly birational to a $Q$-abelian variety $Y$.
Thus statement follows from \cref{generically finite,Q-abelian}.

(2) Since $K_X\equiv0$ and $X$ and $G$ satisfies {\em \cref{hyp}}$(n, n-2)$, then by \cite[Theorem 1.2]{HL19},  after replacing $G$ by a finite-index subgroup, there is a $G$-equivariant quasi-\'etale morphism $Y\to X$, such that $Y$ is $G$-equivariantly birational to 
either a weak Calabi--Yau variety, an abelian variety, or a product of a weak Calabi--Yau surface and an abelian variety.
In the last case automorphisms on the product are split, i.e. they are products of automorphisms on each factor (\cite[Lemma 2.14]{NZ10}).
As for endomorphisms on surfaces, \cref{KSC} is proved in \cite[Theorem 1.3]{MZ19}.
Thus the statement follows from  \cref{lemma:lifting,generically finite,Q-abelian,Sano}.
\end{proof}

\subsection{Albanese morphisms}

To prove KSC, we can assume that the albanese morphism is surjective due to the following proposition.

\begin{prop}\label{alb}
Let $f$ be a surjective endomorphism of a normal projective variety $X$. 
If $f$ has Zariski dense orbits, then the Albanese morphism $\pi : X\to A$ is surjective.
\end{prop}

\begin{proof}
Let $Y=\pi(X)$. Suppose $Y \neq A$.
Let $B$ be the identity component of the stabilizer $\mathrm{Stab}(Y)$ of $Y$ in $A$. 
Since $Y \neq A$, we have $B \neq A$.
Then $A/B$ is a positive dimensional abelian variety and the image $Z$ of $Y$ in $A/B$
has finite stabilizer in $A/B$.
Then $Z$ is of general type (cf. \cite[Theorem 3.7]{Mori85}).
Note that $\dim Z>0$ because it generates $A/B$. 

Now, by the universality of the albanese morphism, $f$ induces an endomorphism $T_{a}\circ g$ on $A$,
where $T_{a}$ is the translation by an element $a \in A$ and $g$ is a group endomorphism of $A$.
Then for any $b \in B$, we have 
\[
Y+g(b)=(T_{a}\circ g)(Y)+g(b)=g(Y)+g(b)+a=g(Y+b)+a=g(Y)+a=Y.
\]
Therefore, we have $g(B)\subset B$.
Thus, $g$ induces an group endomorphism $ \overline{g}$ on $A/B$.
Then $f$ induces $T_{a}\circ g$ on $A$ and $T_{ \overline{a}}\circ \overline{g}$ on $A/B$, where $ \overline{a}$ is the image of $a$ in $A/B$.
Since $Z$, which is the image of $X$ in $A/B$, is of general type,  $T_{ \overline{a}}\circ \overline{g}$ restricted on $Z$ is a finite order automorphism.
This contradicts to the fact that $f$ has Zariski dense orbits.

\end{proof}

\section{Projective bundles}\label{projbdl}

\begin{prop}
Let $Y$ be a smooth Fano variety over $ \overline{\bbQ}$ of Picard number one.
Let $X=\bbP_{Y}( \mathcal{E})$ be a projective bundle over $Y$.
Then KSC holds for all surjective endomorphisms $f \colon X \longrightarrow X$.
\end{prop}

\begin{proof}
Let $\pi \colon X \longrightarrow Y$ be the projection.

Step 1.
Replacing $f$ with its iterate, we may assume that $f$ induces an endomorphism $g \colon Y \longrightarrow Y$ such that 
$g\circ \pi =\pi \circ f$ (cf. discussion before Theorem 2 in \cite{Am03}).

Step 2.
Since the Picard number of $Y$ is one, $g$ is polarized and KSC holds for $g$.
Therefore, we may assume $\delta_{f}>\delta_{g}$.
If $\delta_{g}>1$, then $f$ is an int-amplified endomorphism.
Since $X$ is smooth rationally connected, KSC holds for $f$ by \cite[Theorem 1.1]{MY19}.

Step 3.
Suppose $\delta_{f}>\delta_{g}=1$.
We claim that $ \mathcal{E}$ is a direct sum of invertible sheaves, i.e. $ \mathcal{E}\simeq \bigoplus_{i} \mathcal{L}_{i}$ for some $ \mathcal{L}_{i}$.
If $\dim Y=1$, then $Y\simeq \bbP^{1}$ and the claim follows from Grothendieck's theorem.
Suppose $\dim Y>1$. By Kodaira vanishing and Serre duality, $H^{1}(Y, \mathcal{L})=0$ for all invertible sheaves $ \mathcal{L}$ on $Y$
(we use the assumption that the Picard number of $Y$ is one).
Also, $Y$ is simply connected.
We can apply \cite[Theorem 2]{AK17} so that we get the claim.

Now,  $g$ is an automorphism because $Y$ has Picard number one and $\delta_{g}=1$.
Since $Y$ is a smooth Fano, we have $\Pic Y=\bbZ$.
Therefore $g^{*}$ acting on $\Pic Y$ as identity, and we have
$g^{*} \mathcal{E}\simeq \bigoplus_{i} g^{*}\mathcal{L}_{i} \simeq \bigoplus_{i} \mathcal{L}_{i} \simeq \mathcal{E}$.
Thus we get the following commutative diagram:
\[
\xymatrix{
X \ar[rd]^{F} \ar@/_30pt/[rrdd]_{\pi} \ar@/^10pt/[rrrd]^{f}&&& \\
   & X \ar[r]^(.35){h} \ar[rd]_{\pi}& \bbP_{Y}(g^{*} \mathcal{E}) \ar[r] \ar[d] & X \ar[d]^{\pi} \\
   &&Y \ar[r]_{g} & Y
}
\]
where $h$ is the isomorphism induced by the isomorphism $ \mathcal{E} \simeq g^{*} \mathcal{E}$
and $F$ is the morphism induced by the universal property of fiber product.
Since $f$ has degree lager than one on the fibers, so does $F$.
By \cite[Theorem 1]{Am03} and the comment below it and simply connectedness of $Y$, we get $X \simeq Y \times \bbP^{r}$ where $r+1$ is the rank of $ \mathcal{E}$.
By \cite[Theorem 4.6]{Sano16}, 
and KSC for polarized endomorphisms, we are done.
\end{proof}

\section{Canonical heights}\label{canht}

\begin{prop}
Let $X$ be a geometrically integral variety over a number field $K$ and 
suppose $X_{ \overline{K}}$ is $\bQ$-factorial normal projective variety.
Let $D$ be an $\bR$-divisor on $X$.
Suppose $D$ is $\bR$-linearly equivalent to at least two effective $\bR$-divisors, i.e.
there exist effective $\bR$-divisors $D'$ and $D''$ such that $D \sim_{\bR} D' \sim_{\bR} D''$ and $D' \neq D''$.
Let $h_{D}$ be a Weil height function associated with $D$ and $B,d$ positive real numbers.
Then the set
\begin{equation*}
 \bigg\{ P\in X(L) \bigg| h_{D}(P)\le B, K\subseteq L\subseteq \overline{K} \text{ is an intermediate field with } [L: K]\le d \bigg\}
\end{equation*}
is not Zariski dense in $X_{\overline{K}}$.
\end{prop}

\begin{proof}
Step 1.
There is an effective $\bbR$-divisor $\sum_{i=1}^{r} c_{i}F_{i}$ where $r\geq1$ and $F_{i}$ are prime divisor such that
$D\sim_{\bbR} \sum_{i=1}^{r} c_{i}F_{i}$, $c_{1}\neq0$, and $F_{1}$ has positive Iitaka dimension.

Let $D'=\sum a_{i}E_{i}$ and $D''=\sum b_{i}E_{i}$ where $E_{i}$'s are prime divisors.
Set $D_{0}=\sum \min\{a_{i}, b_{i}\}E_{i}$.
Then $D-D_{0}\sim_{\bbR} D'-D_{0} \sim_{\bbR} D''-D_{0}$ and 
two effective divisors $D'-D_{0}$ and $D''-D_{0}$ have no common components.
By \cite[Proposition 3.5.4]{BCHM10}, $D-D_{0} \sim_{\bbR} M$ where $M$ is an effective $\bbR$-divisor such that
every component is movable.

Step 2.
By step one, $h_{D}=\sum c_{i}h_{F_{i}}+O(1)\geq c_{1}h_{F_{1}}+O(1)$ on $(X_{ \overline{K}}\setminus \bigcup_{i\geq2}\Supp F_{i})( \overline{K})$.
Now the proposition follows from \cite[Proposition 3.5]{Matsuzawa19}.

\end{proof}

\begin{prop}
\text{(cf. \cite[Proposition 3.6]{Matsuzawa19})}
Let $X$ be a $\bQ$-factorial normal projective variety and $f \colon X \longrightarrow X$ be a surjective morphism with $\delta_{f}>1$.
\begin{enumerate}
\item If there is an $\bR$-divisor $D$ which is $\bR$-linearly equivalent to at least two effective $\bR$-divisors
such that $f^{*}D \sim_{\bR} \delta_{f}D$, then KSC holds for $f$.
\item Suppose $f$ is an automorphism. 
Suppose further that there are $\bR$-divisors $D_{+}$ and $D_{-}$ such that $f^{*}D_{+} \sim_{\bR} \delta_{f}D_{+}$, $(f^{-1})^{*}D_{-} \sim_{\bR} \delta_{f^{-1}}D_{-}$
and $D_{+}+D_{-}$ is $\bR$-linearly equivalent to at least two effective $\bR$-divisors.
Then KSC holds for $f$.
\end{enumerate}
\end{prop}

\begin{rmk}
The condition in (2) is a generalization of property (B) in \cite[Definition 3.7]{LS18-IMRN}.
\end{rmk}

\begin{proof}
(1)
Take a height function $h_{D}$ associated with $D$.
Let $\hat{h}(x)=\lim_{n\to \infty}h_{D}(f^{n}(x))/\delta_{f}^{n}$ for all $x\in X( \overline{K})$.
Then $\hat{h}$ satisfies:
\begin{itemize}
\item $\hat{h}(f(x))=\delta_{f}\hat{h}(x)$;
\item $\hat{h}=h_{D}+O(1)$;
\item The set of points $x\in X( \overline{K})$ such that $\hat{h}(x)=0$ and $x$ is defined over a single number field is not Zariski dense.
\end{itemize}
This implies that if $x\in X( \overline{K})$ has Zariski dense $f$-orbit, then $\hat{h}(x)>0$ and $ \alpha_{f}(x)=\delta_{f}$.

(2)
As in (1), let
\begin{align*}
\hat{h}_{D_{+}}(x)=\lim_{n\to \infty}\frac{h_{D_{+}}(f^{n}(x))}{\delta_{f}^{n}}\\
\hat{h}_{D_{-}}(x)=\lim_{n\to \infty}\frac{h_{D_{-}}(f^{-n}(x))}{\delta_{f^{-1}}^{n}}
\end{align*}
for all $x\in X( \overline{K})$.
Then 
\begin{itemize}
\item $\hat{h}_{D_{+}}(f(x))=\delta_{f}\hat{h}_{D_{+}}(x)$, $\hat{h}_{D_{-}}(f(x))=\delta_{f^{-1}}^{-1}\hat{h}_{D_{-}}(x)$;
\item $\hat{h}_{D_{+}}=h_{D_{+}}+O(1)$, $\hat{h}_{D_{-}}=h_{D_{-}}+O(1)$;
\item The set of points $x\in X( \overline{K})$ on which $\hat{h}_{D_{+}}+\hat{h}_{D_{-}}$ is bounded
 and $x$ is defined over a single number field is not Zariski dense.
\end{itemize}
This implies that if $x\in X( \overline{K})$ has Zariski dense $f$-orbit, then $\hat{h}_{D_{+}}(x)>0$ and $ \alpha_{f}(x)=\delta_{f}$.

\end{proof}

\section{Augmented base loci and endomorphisms}\label{augb}

The {\it augmented base locus} of an $\bbR$-Cartier divisor $D$ is defined as follows.

\begin{defn}
The augmented base locus of an $\mathbb R$-Cartier divisor $D$ on a normal projective variety $X$ is the Zariski closed subset
\begin{equation*}
                                 \mathbf{B}_+(D):=\bigcap \mathbf{B}(D-A),
\end{equation*}
where $\mathbf{B}(-)$ stands for the stable base locus and
the intersection is taken over all $\bbR$-Cartier ample divisors $A$ such that $D-A$ is $\bbQ$-Cartier.
\end{defn}

The augmented base locus $\mathbf{B}_{+}(D)$ is equal to the exceptional locus $\mathbb{E}(D)$ if $D$ is nef (cf. \cite{Bir17}).

\begin{prop}\label{f-invariance}
Let $f: X\to X$ be a surjective endomorphism and $D$ an $\mathbb R$-Cartier divisor on a normal projective variety $X$.
If $f^*D\equiv d D$ with $d \in \bbR_{>0}$, then $f(\mathbf{B}_+(D))\subset \mathbf{B}_+(D)$.
Moreover, if $D$ is nef, then $f^{-1}(\mathbf{B}_+(D))=\mathbf{B}_+(D)$.
\end{prop}

\begin{proof}
Suppose $x\notin f^{-1}(\mathbf{B}_+(D))$.
Then there is an $\bbR$-Cartier ample divisor $D$ on $X$ such that $D-A$ is $\bbQ$-Cartier and $f(x)\notin \mathbf{B}(D-A)$.
Then $x \notin \mathbf{B}(f^{*}D-f^{*}A)$, and this implies $x \notin \mathbf{B}_+(f^{*}D)$.
Since the augmented base locus depends only on the numerical class of the divisor and it does not change by positive real multiplication of the divisor,
we have 
$x \notin \mathbf{B}_+(f^{*}D)= \mathbf{B}_+(dD)=\mathbf{B}_+(D)$.
This shows the first statement.

Now assume $D$ is nef.
Since the exceptional locus is compatible with pull-backs by finite surjective morphisms, we have
\begin{align*}
f^{-1}(\mathbf{B}_+(D))=f^{-1}(\mathbb{E}(D))= {\mathbb{E}(f^{*}D)}=\mathbf{B}_+(f^{*}D)=\mathbf{B}_+(dD)=\mathbf{B}_+(D).
\end{align*}

\end{proof}

Let $X$ be a normal projective variety and $f \colon X \longrightarrow X$ be a surjective endomorphism.
Suppose there exits a nef and big $\bbR$-Cartier divisor $D$ such that $f^{*}D\equiv dD$ with $d\in \bbR_{>1}$.
By \cite{MZ18}, $f$ is actually polarized and we can construct canonical heights associated with ampler divisors.
But we may also consider the canonical height $\hat{h}_{D}(x)=\lim_{n \to \infty}h_{D}(f^{n}(x))/d^{n}$ associated with $D$ and this has the following property:
If $x \in (X\setminus \mathbf{B}_+(D))(\overline{\mathbb Q})$, then $x$ is $f$-preperiodic if and only if $\hat{h}_{D}(x)=0$.

Let us give a purely geometric application of \cref{f-invariance}.

\begin{prop}\label{psef=nef}
Let $X$ be a $\bbQ$-factorial normal projective variety and $f \colon X \longrightarrow X$ be a surjective endomorphism 
such that the map $f^{*} \colon N^{1}(X)_{\bbR} \longrightarrow N^{1}(X)_{\bbR}$ is a scalar multiplication.
Suppose further that $f$ has no non-trivial totally invariant closed subset (i.e. $f^{-1}(Z)=Z$ implies $Z=X$ or $Z=\emptyset$).
Then the pseudo-effective cone of $X$ is equal to the nef cone of $X$: $ \Eff(X)=\Nef(X)$.
\end{prop}
\begin{proof}
By \cref{f-invariance} and the assumption, for every nef and big $\bbR$-Cartier divisor $D$ on $X$, we have $\mathbf{B}_+(D)=\emptyset$.
That is, every nef and big $\bbR$-Cartier divisor is ample.
This implies all big $\bbR$-Cartier divisor is nef, and therefore ample.
\end{proof}

\begin{cor}
Let $X$ be a smooth rationally connected variety and $f \colon X \longrightarrow X$ be a polarized endomorphism
(i.e. $f^{*}H\equiv dH$ for some ample $\bbR$-divisor $H$ and $d\in \bbR_{>1}$).
Suppose $f$ has no non-trivial totally invariant closed subset.
Then $X$ is a Fano variety with Picard number at most $\dim X$.
\end{cor}
\begin{proof}
By \cite{MZ18}, if we replace $f$ with its iterate, $f^{*}$ acts as a scalar multiplication on $N^{1}(X)_{\bbR}$.
By \cref{psef=nef}, we have $ \Eff(X)=\Nef(X)$.
By \cite[Corollary 1.4]{Y20}, $X$ is of Fano type, i.e. $-K_{X}- \Delta$ is ample for some effective $\bbQ$-divisor $ \Delta$.
Thus, we get $-K_{X}$ is ample.

Every $K_{X}$-negative contraction of $X$ is of fiber type because $f$ has no non-trivial totally invariant closed subset.
Moreover $f$ induces a polarized endomorphism on the base of the contraction which has no non-trivial totally invariant closed subset.
We can repeat this process and final out put is a point (cf. \cite{MZ18}).
This implies that the Picard number of $X$ is less than or equal to $\dim X$.
\end{proof}

\section*{Acknowledgments}
The first author was partially supported by the China Scholar Council `High-level university graduate program'.
The second author is supported by JSPS Overseas Research Fellowship.

\end{document}